\documentclass[preprint, 12pt]{elsarticle}
\usepackage{amssymb}
\usepackage{amsthm}
\usepackage{lineno}
\newtheorem{theorem}{Theorem}[section]
\newtheorem{lemma}[theorem]{Lemma}

\newtheorem{remark}[theorem]{Remark}
\newtheorem{proposition}[theorem]{Proposition}

\theoremstyle{definition}

\journal{**}
\begin{document}
\begin{frontmatter}

\title{The structure of strong $k$-quasi-transitive digraphs with large diameters
\tnoteref{label1}}\tnotetext[label1]{This work is supported by the National Natural Science Foundation for Young Scientists of China (11401354)(11501490)(11501341).}
\author{Ruixia Wang\corref{cor1}}
\ead{wangrx@sxu.edu.cn}\cortext[cor1]{Corresponding  author.}
\author{Hui Zhang}
\address{School of Mathematical Sciences, Shanxi University, Taiyuan, Shanxi,
              PR China\fnref{label3}}
\begin{abstract}
Let $k$ be an integer with $k\geq 2$. A digraph $D$ is $k$-quasi-transitive, if for any path $x_0x_1\ldots x_k$ of length $k$,
$x_0$ and $x_k$ are adjacent. Suppose that there exists a path of length at least $k+2$ in $D$.
Let $P$ be a shortest path of length $k+2$ in $D$. Wang and Zhang [Hamiltonian paths in $k$-quasi-transitive digraphs,
Discrete Mathematics, 339(8) (2016) 2094--2099] proved that if $k$ is even and $k\ge 4$, then $D[V(P)]$ and $D[V(D)\setminus V(P)]$ are both semicomplete digraphs.  In this paper, we shall prove that if $k$ is odd and $k\ge 5$, then $D[V(P)]$ is either
a semicomplete digraph or a semicomplete bipartite digraph and $D[V(D)\setminus V(P)]$
is either a semicomplete digraph, a semicomplete bipartite digraph or an empty digraph.
\end{abstract}

\end{frontmatter}

\section{Terminology and introduction}

We shall assume that the reader is familiar with the standard terminology on
digraphs and refer the reader to \cite{bang2} for terminology not defined here.
We only consider finite digraphs without loops or multiple arcs. Let $D$ be a
digraph with vertex set $V(D)$ and arc set $A(D)$.
For any $x,y\in V(D)$, we will write $x\rightarrow y$ if $xy\in
A(D)$, and also, we will write $\overline{xy}$ if $x\rightarrow y$ or $y\rightarrow x$.
For disjoint subsets $X$ and $Y$ of $V(D)$, $X\rightarrow Y$ means
that every vertex of $X$ dominates every
vertex of $Y$, $X\Rightarrow Y$ means that there is no arc from
$Y$ to $X$ and $X\mapsto Y$ means that both of $X\rightarrow Y$
and $X\Rightarrow Y$ hold. For subsets $X,Y$ of $V(D)$, we define $(X,Y)=\{xy\in  A(D):  x\in X, y\in Y\}$.
Let $H$ be a subset of $V(D)$ and $x\in V(D)\setminus H$.
We say that $x$ and $H$ are adjacent if $x$ and some vertex of $H$ are adjacent.  For $S\subseteq V(D)$,
we denote by $D[S]$ the subdigraph of $D$ induced by the vertex set $S$. The converse of $D$,  $\overleftarrow{D}$, is the digraph which is obtained from $D$ by reversing all arcs.

Let $x$ and $y$ be two vertices of $V(D)$. The {\it distance}
from $x$ to $y$ in $D$, denoted $d(x,y)$, is the minimum length of an $(x, y)$-path,
if $y$ is reachable from $x$, and otherwise $d(x, y)=\infty$.
The distance from a set $X$ to a set $Y$ of vertices in $D$ is
$d(X, Y ) = \max\{d(x, y): x\in X, y\in Y\}$. The {\it diameter} of $D$ is diam$(D)= d(V(D),V(D))$.
Clearly, $D$ has finite diameter if and only if it is strong.

Let $P=y_0y_1\ldots y_m$ be a path or a cycle of $D$. For $i<j$, $y_i, y_j\in V(P)$
we denote the subpath $y_iy_{i+1}\ldots y_j$ of $P$ by $P[y_i, y_j]$.
Let $Q=q_0q_1\ldots q_n$ be a vertex-disjoint  path or  cycle with $P$ in $D$.
If there exist $y_i\in V(P)$ and $q_j\in V(Q)$ such that $y_iq_j\in A(D)$,
then we will use $P[y_0, y_i]Q[q_j,q_n]$ to denote the path $y_0y_1\ldots y_iq_jq_{j+1}\ldots q_n$.

A digraph is quasi-transitive, if for any path $x_0x_1x_2$ of length 2, $x_0$ and $x_2$ are adjacent.
The concept of $k$-quasi-transitive digraphs was introduced in \cite{cruz}
as a generalization of quasi-transitive digraphs. A digraph is {\it $k$-quasi-transitive},
if for any  path $x_0x_1\ldots x_k$ of length $k$, $x_0$ and $x_k$ are adjacent.
The $k$-quasi-transitive digraph has been studied in \cite{cruz2013, cruz, wang3, wang2, wang5}.

The following theorem completely characterizes strong quasi-transitive digraphs
in recursive sense.

\begin{theorem}\cite{bang1}
Let $D$ be a strong quasi-transitive digraph. Then there exists a strong semicomplete digraph $S$ with vertices
 $\{v_1, v_2,\ldots, v_s\}$ and quasi-transitive digraphs $\{Q_1, Q_2, \ldots, Q_s\}$ such
that $Q_i$ is either a vertex or is non-strong and $D=S[Q_1,Q_2, \ldots, Q_s]$,
where $Q_i$ is substituted for $v_i$, $i=1, 2,\ldots, s$.
\end{theorem}

In \cite{galeana2010}, Galeana-S\'{a}nchez et al. characterized strong 3-quasi-transitive digraphs. Let $F_n$ be a digraph with vertex set $\{x_0, x_1, \ldots, x_n\}$ and arc set $\{x_0x_1, x_1x_2, x_2x_0\}\cup \{x_ix_1, x_0x_i, i=3,4,\ldots, n\}$, where $n\ge 3$.

\begin{theorem}\cite{galeana2010} Let $D$ be a strong 3-quasi-transitive digraph of order $n$. Then $D$ is either a semicomplete digraph, a semicomplete bipartite digraph, or isomorphic to $F_n$.
\end{theorem}

Let $D$ be a strong $k$-quasi-transitive digraph with diam$(D)\ge k+2$. Then there exist two vertices $u,v\in V(D)$ such that $d(u,v)=k+2$. Let $P$ be a shortest $(u,v)$-path in $D$.
In \cite{wang5}, Wang and Zhang proved that if $k(\ge 4)$ is even, then $D[V(P)]$ and
$D[V(D)\setminus V(P)]$ are both semicomplete digraphs. In Section 2, we prove the following theorem.

\begin{theorem}\label{main result}
  Let $k$ be an odd integer with $k\ge 5$, $D$ be a strong $k$-quasi-transitive digraph with diam$(D)\ge k+2$ and $u,v\in V(D)$ such that $d(u,v)=k+2$. If $P$ is a shortest $(u,v)$-path, then $D[V(P)]$ is either
a semicomplete digraph or a semicomplete bipartite digraph and $D[V(D)\setminus V(P)]$
is either a semicomplete digraph, a semicomplete bipartite digraph or an empty digraph.
\end{theorem}

\section{Main results}

We begin with a rather trivial observation.

\begin{remark} \label{remark} A  digraph $D$ is $k$-quasi-transitive if and only if $\overleftarrow{D}$ is $k$-quasi-transitive.
\end{remark}

\begin{proposition} \label{back} \cite{cruz2013}   Let $k$ be an integer with $k\ge 2$, $D$ be a $k$-quasi-transitive digraph
and $u,v\in V(D)$ such that $d(u,v)=k+2$.
Suppose that $P=x_0x_1\ldots x_{k+2}$ is a shortest $(u,v)$-path, where $u=x_0$ and $v=x_{k+2}$. Then each of the following holds:

\begin{description}
\item [(1)]  $x_{k+2}\rightarrow x_{k-i}$, for every odd $i$ such that $1\le i\le k$.
\item [(2)] $x_{k+1}\rightarrow x_{k-i}$, for every even $i$ such that $1\le i\le k$.
\end{description}
\end{proposition}

From the above proposition, we have seen that for a longer path, there are some structural property in
a $k$-quasi-transitive digraph. There exist two vertices $u, v$ such that $d(u, v)=k+2$ in a
strong digraph $D$ with diam$(D)\ge k+2$. So for the rest of this paper, let $k$ be an odd integer with $k\ge 5$,
$D$ denote a strong $k$-quasi-transitive digraph with diam$(D)\ge k+2$ and $P=x_0x_1\ldots x_{k+2}$
denote a shortest $(u, v)$-path in $D$, where $u=x_0$ and $v=x_{k+2}$. Denote $O(P)=\{x_1, x_3,\ldots, x_{k+2}\}$ and
$E(P)=\{x_0, x_2,\ldots, x_{k+1}\}$.  First we show the following structural property on $D[V(P)]$.

\begin{lemma}\label{bipartite} Let $x_t$ and $x_s$ be two arbitrary vertices in $V(P)$ with $s>t$.
The following three statements hold.
\begin{description}
\item[(1)] $D[V(P)]$ contains a semicomplete bipartite digraph as its subdigraph with bipartition $(O(P), E(P))$.

\item[(2)] There is a path of length $k-2$ from $x_s$ to $x_t$ in $D[V(P)]$ when the parity of $s$ and $t$ are different.
Moreover, for any $x, y\in V(D)\setminus V(P)$, if $y\rightarrow x_s$ and $x_t\rightarrow x$, then $\overline{xy}$.

\item[(3)] There is a path of length $k-1$ from $x_s$ to $x_t$ in $D[V(P)]$ when the parity of $s$ and $t$ are same.
Moreover, for any $x\in V(D)\setminus V(P)$, if $x\rightarrow x_s$, then $\overline{xx_t}$ and $x\rightarrow x_t$ if $s\ge t+4$; if $x_t\rightarrow x$,
then $\overline{xx_s}$ and $x_s\rightarrow x$ if $s\ge t+4$.
\end{description}
\end{lemma}

\begin{proof} (1) It suffices to show that, for any $x_i\in E(P)$ and $x_j\in O(P)$,  $\overline{x_ix_j}$.
By Proposition \ref{back}, $x_{k+2}\rightarrow \{x_0, x_2,\ldots, x_{k-1}\}$ and $x_{k+1}\rightarrow \{x_1, x_3, \ldots, x_{k-2}\}$.
So we assume $0\le i\le k-1$ and $1\le j\le k$.
Clearly, $\overline{x_{j-1}x_j}$ and $\overline{x_{j+1}x_j}$.
Now assume that $i\le j-3$ or $i\ge j+3$.
First consider the case $0\le i\le j-3$.
For $i=0$, we proof $x_j\rightarrow x_0$ by induction on $j$.
The length of the path $P[x_3, x_{k+2}]x_0$ is $k$, which implies that $\overline{x_3x_0}$ and further $x_3\rightarrow x_0$ since $P$ is minimal.
For the inductive step, let us suppose that $x_j\rightarrow x_0$ for $3\le j\le k-2$.
The path $P[x_{j+2}, x_{k+1}]P[x_1, x_j]x_0$ implies $\overline{x_{j+2}x_0}$ and further $x_{j+2}\rightarrow x_0$.
For $i\ge 2$, the path $P[x_j, x_{k+1}]P[x_{i+1}, x_{j-2}]P[x_0,x_i]$ implies $\overline{x_ix_j}$.
Now consider the case $i\ge j+3$. By the above argument, we know that $x_{i-1}\rightarrow x_0$.
The path $P[x_i, x_{k+1}]P[x_{j+2}, x_{i-1}]P[x_0,x_j]$ implies $\overline{x_ix_j}$.
Therefore, $D[V(P)]$ contains a semicomplete bipartite digraph as its subdigraph with bipartition $(O(P),E(P))$.

(2) Assume $x_s\in E(P)$ and $x_t\in O(P)$. If $s-t=1$, then $P[x_s, x_p]P[x_{p-k+2}, x_t]$ is a path of length $k-2$,
where $p=k+1$ when $s=k+1$ and $p=k-1$ when $s<k+1$.
If $s-t=3$, then $P[x_s, x_{k+1}]P[x_1, x_t]$ is a path of length $k-2$.
If $s-t\ge 5$, then $P[x_s, x_{k+1}]P[x_{t+2}, x_{s-2}]P[x_1, x_t]$ is a path of length $k-2$.
Analogously, if $x_s\in O(P)$ and $x_t\in E(P)$, we can find the desired path.

Denote by $Q$ the path  from $x_s$ to $x_t$ of length $k-2$ in $D[V(P)]$.
Let $x, y\in V(D)\setminus V(P)$.
If $y\rightarrow x_s$ and $x_t\rightarrow x$, then the path $yQx$ implies $\overline{xy}$.

(3) Assume $x_s, x_t\in O(P)$.
If $s\le k$, then,  by (2), there is a path $Q$ of length $k-2$ from $x_{s+1}$ to
$x_t$ in $D[V(P)]$.
By the proof of (2), observe that we can find such a path $Q$ so that $x_s\notin V(Q)$.
So $x_sQ$ is the desired path. Now assume $s=k+2$. For $t=k$,  $x_{k+2}P[x_2, x_k]$ is the desired path.
For $t\in \{1, 3, \ldots, k-2\}$,  $x_{k+2}P[x_{t+1}, x_{k-1}]P[x_1, x_t]$ is the desired path.
Analogously, if $x_s, x_t\in E(P)$, we can find the desired path.

Denote by $R$ the path  from $x_s$ to $x_t$ of length $k-1$ in $D[V(P)]$. Let $x\in V(D)\setminus V(P)$ be arbitrary.
If  $x\rightarrow x_s$, then the path $xR$ implies $\overline{xx_t}$ and $x\rightarrow x_t$ if $s\ge t+4$ as $P$ is minimal;
if $x_t\rightarrow x$, then the path $Rx$ implies $\overline{xx_s}$ and $x_s\rightarrow x$ if $s\ge t+4$ as $P$ is minimal.
\end{proof}


\begin{lemma} \label{semicomplete} If there exist two vertices $x_i,  x_j\in E(P)$ or $x_i,  x_j\in O(P)$
such that $\overline{x_ix_j}$, then $D[V(P)]$ is a semicomplete digraph and for any $1\le t+1<s\le k+2$, $x_s\rightarrow x_t$.
\end{lemma}

\begin{proof} Without loss of generality, assume that $i<j$. Thus, $x_j\rightarrow x_i$.
By Lemma \ref{bipartite}, every vertex of $E(P)$ is adjacent to every vertex of $O(P)$,
that is to say, for any $x_{\alpha}$ and $x_{\beta}$ such that $\alpha$ and $\beta$ have different parity,
we have $\overline{x_{\alpha}x_{\beta}}$, in particular, if $\alpha >\beta+1$, then $x_{\alpha}\rightarrow x_{\beta}$
by the minimality of $P$.

First we show that $x_{k+2}\rightarrow \{x_1, x_3, \ldots, x_k\}$.
We claim that for some $q\ge 3$ if $x_{k+2}\rightarrow x_{q}$,
then $x_{k+2}\rightarrow x_{q-2}$. In fact, the length of the path $x_{k+2}P[x_q, x_{k+1}]P[x_1, x_{q-2}]$ is $k$.
So $\overline{x_{k+2}x_{q-2}}$ and $x_{k+2}\rightarrow x_{q-2}$. Next we claim that for some $p\ge 3$
if $x_p\rightarrow x_1$, then $x_{p+2}\rightarrow x_1$ (if $x_{p+2}$ exists).
In fact, the length of the path $P[x_{p+2}, x_{k+2}]P[x_2, x_p]x_1$ is $k$ and
so $x_{p+2}\rightarrow x_1$. Suppose that $x_i, x_j\in O(P)$. If $j=k+2$ or $i=1$, then, by above two claims,
we can obtain that $x_{k+2}\rightarrow x_1$. Now assume that $j\le k$ and $i\ge 3$.
Note that the length of the path $P[x_{j+2}, x_{k+2}]P[x_{i+1}, x_j]x_iP[x_0, x_{i-2}]$ is $k$.
So $x_{j+2}\rightarrow x_{i-2}$. Repeating this way and using above two claims,
we can obtain $x_{k+2}\rightarrow x_1$. Suppose that $x_i, x_j\in E(P)$.
Analogously, we can obtain $x_{k+1}\rightarrow x_0$. Note that the length of the path $x_{k+2}P[x_4, x_{k+1}]x_0x_1$ is $k$.
So $x_{k+2}\rightarrow x_1$.
Next we show $x_{k+2}\rightarrow x_{k-p}$ for $0\le p\le k-3$ by induction on even $p$.
For $p=0$, the path $x_{k+2}P[x_1, x_k]$ implies $\overline{x_kx_{k+2}}$ and $x_{k+2}\rightarrow x_k$.
For the inductive step, let us suppose that $x_{k+2}\rightarrow x_{k-p}$ for even $p$ and
$0\le p\le k-5$. Then the path of $x_{k+2}P[x_{k-p}, x_{k+1}]P[x_1, x_{k-(p+2)}]$ implies
$\overline{x_{k+2}x_{k-(p+2)}}$ and $x_{k+2}\rightarrow x_{k-(p+2)}$. From now on, we have shown that $x_{k+2}\rightarrow \{x_1, x_3, \ldots, x_k\}$.

Now we show that for any $1\le t+1<s\le k+2$, $x_s\rightarrow x_t$.
It suffices to show that, for any $x_t, x_s\in O(P)$ or $x_t, x_s\in E(P)$,
there exists a path of length $k$ from $x_s$ to $x_t$ since $P$ is minimal.
Suppose that $x_s, x_t\in O(P)$. If $s=k+2$, then we are done. Assume $s\le k$. If $s-t=2$,
then the length of the path $P[x_s, x_{k+2}]P[x_1, x_t]$  is $k$. If $s-t\ge 4$,
then the length of the path $P[x_s, x_{k+2}]P[x_{t+2}, x_{s-1}]P[x_1, x_t]$  is $k$.
Suppose that $x_s, x_t\in E(P)$. First consider $s=k+1$.  If $t=0$,
then the length of the path $x_{k+1}x_{k+2}P[x_3, x_k]x_0$  is $k$. If $t\ge 2$,
then the length of the path $x_{k+1}x_{k+2}P[x_{t+1}, x_k]P[x_2, x_t]$ is $k$.
Now consider that $s\le k-1$. If $s-t=2$, then the length of the path $P[x_s, x_{k+1}]P[x_0, x_t]$  is $k$.
If $s-t\ge 4$,  then the length of the path $P[x_s, x_{k+1}]P[x_{t+2}, x_{s-1}]P[x_0, x_t]$ is $k$.

From now on, we have shown that $D[V(P)]$ is a semicomplete digraph and $x_s\rightarrow x_t$ for $1\le t+1<s\le k+2$.\end{proof}

According to Lemmas \ref{bipartite} and \ref{semicomplete}, we can easily obtain the following theorem.

\begin{theorem} \label{theorem} The digraph $D[V(P)]$ is either a semicomplete digraph or a semicomplete bipartite digraph.
\end{theorem}

In the rest of the paper, we study the structure of $D[V(D)\setminus V(P)]$.

\begin{lemma}\label{outsemicomplete}  Let $H$ be a digraph and $u, v\in V(H)$ such that $d(u,v)=n$ with $n\ge 4$ in $H$. Let $Q=x_0x_1\ldots x_n$ be a shortest $(u,v)$-path in $H$.
If $H[V(Q)]$ is a semicomplete digraph, then, for any  $x_i, x_j\in V(Q)$ with $0\le i<j\le n$,
there exists a path of length $p$ from $x_j$ to $x_i$ with $p\in \{2, 3,\ldots, n-1\}$ in $H[V(Q)]$.
\end{lemma}

\begin{proof}  We proof the result by induction on $n$. For $n=4$, it is not difficult
to check that the result is true. Suppose $n\ge 5$. Assume $j-i=n$. It must be $j=n$ and $i=0$.
Then the length of the path $x_nP[x_2, x_p]x_0$ is $p$, with $p\in \{2,3,\ldots, n-1\}$.
Now assume $1\le j-i\le n-1$. Then $x_i, x_j\in V(H[x_0, x_1,\ldots, x_{n-1}])$ or $x_i, x_j\in V(H[x_1, x_2,\ldots, x_n])$.
Without loss of generality, assume that $x_i, x_j\in V(H[x_0, x_1,\ldots, x_{n-1}])$.
By induction, there exists a path of length $p$ from $x_j$ to $x_i$ with $p\in\{2,3,\ldots, n-2\}$.
Now we only need to show that there exists a path of length $n-1$ from $x_j$ to $x_i$.
If $j-i=1$, then $P[x_j, x_{n-1}]P[x_0, x_i]$ is the desired path. If $j-i=2$,
then $P[x_j, x_n]P[x_0, x_i]$ is the desired path. If $3\le j-i\le n-1$,
then $P[x_j, x_n]P[x_{i+2}, x_{j-1}]P[x_0,x_i]$ is the desired path.
\end{proof}

By Lemma \ref{outsemicomplete}, we can obtain the following lemma.

\begin{lemma}\label{semicompleteadj}   Suppose that $D[V(P)]$ is a semicomplete digraph.
For any $x\in V(D)\setminus V(P)$ and $x_i\in V(P)$, if $x\rightarrow x_i$,
then $x$ and every vertex of $\{x_0, x_1, \ldots, x_{i-1}\}$ are adjacent;
if $x_i\rightarrow x$, then $x$ and every vertex of $\{x_{i+1}, x_{i+2}, \ldots, x_{k+2}\}$ are adjacent.
\end{lemma}

\begin{proof} If  $x\rightarrow x_i$, then for any $x_j\in \{x_0, x_1, \ldots, x_{i-1}\}$,
by Lemma \ref{outsemicomplete}, there exists a  path $Q$ of length $k-1$ from $x_i$ to $x_j$.
Then the path $xQ$ implies $\overline{xx_j}$. If  $x_i\rightarrow x$,
then for any $x_j\in \{x_{i+1}, x_{i+2}, \ldots, x_{k+2}\}$, by Lemma \ref{outsemicomplete},
there exists a  path $R$ of length $k-1$ from $x_j$ to $x_i$. Then the path $Rx$ implies $\overline{xx_j}$.
\end{proof}

\begin{lemma}\cite{wang2}\label{adjacent} Let $k$ be an integer with $k\ge 2$ and
$D$ be a strong $k$-quasi-transitive digraph. Suppose that $C=x_0x_1\ldots x_{n-1}x_0$
is a cycle of length $n$ with $n\geq k$ in $D$. Then for any $x\in V(D)\setminus V(C)$, $x$ and $V(C)$ are adjacent.
\end{lemma}

\begin{lemma} \cite{cruz2013d}\label{cruz2013d} Let $k$ be an integer with $k\ge 2$ and
$D$ be a $k$-quasi-transitive digraph. Suppose that $C=x_0x_1\ldots x_{n-1}x_0$ is a cycle of
length $n$ with $n\geq k$ in $D$.  For any $x\in V(D)\setminus V(C)$,
if $x\rightarrow x_i$ and $x\Rightarrow V(C)$, then $x\rightarrow x_{i+(k-1)}$;
if $x_i\rightarrow x$ and $V(C)\Rightarrow x$, then $x_{i-(k-1)}\rightarrow x$,
where the subscripts are taken modulo $n$.
\end{lemma}


Note that  $x_{k+2}\rightarrow x_0$ and so $P$ contains a cycle $x_0x_1\ldots x_{k+2}x_0$ of length $k+3$.
Combining this with Lemma \ref{adjacent}, every vertex of $V(D)\setminus V(P)$  is adjacent to $V(P)$.
Hence we can divide $V(D)\setminus V(P)$ into three sets: $I=\{x\in V(D)\setminus V(P):   x\Rightarrow V(P)\}$,
$W=\{x\in V(D)\setminus V(P):  V(P)\Rightarrow x\}$ and
$B=V(D)\setminus (V(P)\cup W\cup I)$.
One of $I$, $W$ and $B$ may be empty.

\begin{lemma}\label{Icocbc}  For any $x\in V(D)\setminus V(P)$, the following hold:
\begin{description}
\item [(1)] Suppose $x\in I$.  If there exists a vertex $x_i\in E(P)$ such that $x\rightarrow  x_i$, then $x\mapsto E(P)$;
if there exists a vertex $x_i\in O(P)$ such that $x\rightarrow x_i$, then $x\mapsto O(P)$.

\item[(2)] Suppose $x\in W$.  If there exists a vertex $x_i\in E(P)$ such that $x_i\rightarrow  x$, then $E(P)\mapsto x$;
if there exists a vertex $x_i\in O(P)$ such that $x_i\rightarrow x$, then $O(P)\mapsto x$.

\item[(3)] Suppose $x\in B$.  If $x$ and $O(P)$ are adjacent, then either $x$ and every vertex of $O(P)$ are adjacent
or there exist two vertices $x_s, x_t\in O(P)$ with $3\le t<s\le k$ such that
$\{x_s, x_{s+2}, \ldots, x_{k+2}\}\mapsto x\mapsto \{x_1, x_3, \ldots, x_t\}$;
If $x$ and $E(P)$ are adjacent, then either $x$ and every vertex of $E(P)$ are adjacent
or there exist two vertices $x_s,x_t \in E(P)$ with $2\le t<s\le k-1$
such that $\{x_s, x_{s+2}, \ldots, x_{k+1}\}\mapsto x\mapsto \{x_0, x_2, \ldots, x_t\}$.
\end{description}
\end{lemma}

\begin{proof} (1) First assume that there exists a vertex $x_i\in E(P)$ such that $x\rightarrow  x_i$.
Recall that $x_0x_1\ldots x_{k+2}x_0$ is a cycle of length $k+3$. From now on, the subscripts are taken modulo $k+3$. By Lemma \ref{cruz2013d},
$x\rightarrow x_{i+(k-1)}$.  Without loss of generality, assume $i=0$. Denote $U=\{m(k-1)\in \mathbb{Z}_{k+3}, m\in \mathbb{Z}\}$.
Repeating using Lemma \ref{cruz2013d}, we  can obtain $x\rightarrow  U$. It is easy to see that
$U=\{md: m\in \{0, 1,2, \ldots, \frac{k+3}{d}-1\}\}$, where $d=\gcd(k-1, k+3)$ (here, $\gcd$ means the greatest common divisor.)
Note that $d\ge 2$ as $k$ is odd. So there exist two integers $a$ and $b$ such that $k-1=da$ and $k+3=db$.
Thus, $4=(k+3)-(k-1)=db-da=d(b-a)$.  From this, we know that $d=2$ or $d=4$.

If $d=2$, then $U=\{0, 2, \ldots, k+1\}$. So $x\rightarrow E(P)$ and furthermore $x\mapsto E(P)$ by the definition of $I$. Now assume $d=4$. It is easy to see that
$U=\{0, 4, \ldots, k-1\}$. So $x\mapsto \{x_0, x_4, \ldots, x_{k-1}\}$. Let $x_j\in \{x_2, x_6, \ldots, x_{k-3}\}$ be arbitrary.
Note that $x\mapsto x_{j+2}$. The path $xC[x_{j+2}, x_{k+1}]C[x_1, x_j]$ implies $\overline{xx_j}$ and so $x\mapsto x_j$.
In addition, $xC[x_2, x_{k+1}]$ implies that $\overline{xx_{k+1}}$ and so $x\rightarrow x_{k+1}$. Analogously, if there exists $x_i\in O(P)$
such that $x\rightarrow x_i$, then we can obtain  $x\mapsto O(P)$.

(2) By Remark \ref{remark}, considering the converse of $D$,  the  statement is obvious.

(3) By the definition of $B$, we have $(x, V(P))\neq\emptyset$ and $(V(P), x)\neq\emptyset$. First assume that $x$ and $O(P)$
are adjacent. Without loss of generality, assume that $(x, O(P))\neq\emptyset$, otherwise consider the converse of $D$.
If $x$ and every vertex of $O(P)$ are adjacent, then we are done. Now assume that there exists $x_{n_0}\in O(P)$ such that
$x$ and $x_{n_0}$ are not adjacent. By Lemma \ref{bipartite}(3), $x\Rightarrow \{x_1, x_3, \ldots, x_{n_0-2}\}$
and $\{x_{n_0+2}, x_{n_0+4}, \ldots, x_{k+2}\}\Rightarrow x$. From this with $(x, O(P))\neq\emptyset$, we have $n_0\ge 3$.
Take $t=\max \{i: x\rightarrow x_i \mbox{\ and\  } x_i\in O(P)\}$. Then $1\le t<n_0$. By Lemma \ref{bipartite}(3),
$x$ and every vertex $\{x_1, x_3, \ldots, x_t\}$ are adjacent and furthermore $x\mapsto \{x_1, x_3, \ldots, x_t\}$.
By Lemma \ref{cruz2013d} and $x\rightarrow x_1$, we have $\overline{xx_k}$. If $x\rightarrow x_k$,
then $t\ge k$ and so $n_0=k+2$. But $x\rightarrow x_3$ and Lemma \ref{cruz2013d} implies $\overline{xx_{k+2}}$,
a contradiction. Thus $x_k\mapsto x$.  This together with Lemma \ref{bipartite}(3) implies $\overline{xx_{k+2}}$
and moreover $x_{k+2}\mapsto x$. Take $s=\min \{i: x_i\rightarrow x \mbox{\ and\  } x_i\in O(P)\}$.
Thus $n_0<s\le k$. By Lemma \ref{bipartite}(3), $x$ and every vertex of $\{x_s, x_{s+2}, \ldots, x_{k+2}\}$ are adjacent and
furthermore $\{x_s, x_{s+2}, \ldots, x_{k+2}\}\mapsto x$. Then $x_3x_4\ldots x_{k+2}x$ implies that $\overline{xx_3}$ and so $3<n_0$. This also implies $t\geq 3$. The proof is similar to the above argument when $x$ and $E(P)$ are adjacent. So we omit it.
\end{proof}


\begin{lemma} \label{Btosemi} If $D[V(P)]$ is a semicomplete digraph, then for any $x\in B$,
either $x$ and every vertex of $V(P)$ are adjacent or there exist two vertices $x_t, x_s\in V(P)$
with $4\le t+1<s\le k-1$ such that  $\{x_s, \ldots,  x_{k+2}\}\mapsto x\mapsto \{x_0,\ldots, x_t\}$. \end{lemma}

\begin{proof}  If $x$ and every vertex of $V(P)$ are adjacent, then we are done. Suppose not.
By the definition of $B$,  $(x, V(P))\neq\emptyset$ and $(V(P), x)\neq\emptyset$.
Take $t=\max \{i\ : x\rightarrow x_i\}$ and $s=\min\{j: x_j\rightarrow x\}$.
By Lemma \ref{semicompleteadj}, $x$ and every vertex $V(P[x_0, x_{t-1}])$ are
adjacent and $x$ and every vertex of $V(P[x_s, x_{k+2}])$  are adjacent.
Furthermore, since $x$ and some vertex of $V(P)$ are not adjacent,
we can conclude that $s>t+1$ and $\{x_s, \ldots,  x_{k+2}\}\mapsto x\mapsto \{x_0,\ldots, x_t\}$.
\end{proof}

By Lemmas \ref{outsemicomplete} and \ref{Icocbc}, we can obtain the following lemma.

\begin{lemma}\label{IWtosemi} Suppose that $D[V(P)]$ is a semicomplete digraph. For any $x\in I$, $x\mapsto V(P)$ and
for any $y\in W$, $V(P)\mapsto y$.
\end{lemma}


\begin{lemma}\label{sout}  If $D[V(P)]$ is a semicomplete digraph, then $D[V(D)\setminus V(P)]$ is a semicomplete digraph. \end{lemma}

  \begin{proof} According to Lemmas \ref{Btosemi} and \ref{IWtosemi}, similar to the proof of Lemma 2.11 in \cite{wang5},
  the result can be shown.
  \end{proof}

\begin{lemma}\label{semiDvp}  If $D[V(P)]$ is a semicomplete bipartite digraph, then $D[B]$ is either
a semicomplete bipartite digraph or an empty digraph.  \end{lemma}

\begin{proof}  Now we divide $B$ into two subsets.
Denote $B_1=\{x\in B: x \; \mbox{and }\;  E(P)  \\ \mbox{are adjacent}\}$ and
$B_2=\{x\in B: x \  \mbox{and}  \  O(P)  \ \mbox{are adjacent}\}$.
One of $B_1$ and $B_2$ may be empty.

First we claim $B_1\cap B_2=\emptyset$. Suppose not.  Let $x\in B_1\cap B_2$ be arbitrary. It is not difficult to obtain that
there exist two vertices $x_i$ and $x_j$ in $V(P)$ such that $x_i\rightarrow x\rightarrow x_j$
and the parity of $i$ and $j$ is different.  Without loss of generality, assume that $j$ is even and $i$ is odd.
Take $t=\max\{i: x_i\rightarrow x \  \mbox{and}\  x_i\in O(P)\}$ and $r=\min\{j: x\rightarrow x_j\  \mbox{and}\  x_j\in E(P)\}$.
By Lemma \ref{bipartite}(3), $t=k$ or $k+2$ and $r=0$ or $2$.  If $t=k$, the path $x_{k+1}x_{k+2}P[x_4, x_k]xx_r$
implies that $\overline{x_{k+1}x_r}$. Note that $x_{k+1}, x_r\in E(P)$.  If $t=k+2$, then the path
$P[x_4, x_{k+2}]xx_r$  implies that $\overline{x_4x_r}$. Note that $x_4, x_r\in E(P)$. We have found two vertices of $E(P)$ such that they are adjacent. By Lemma \ref{semicomplete}, $D[V(P)]$ is a semicomplete digraph,
a contradiction. Hence $B_1\cap B_2=\emptyset$.

Now we show that every vertex of $B_1$ and every vertex of $B_2$ are adjacent. Let $x\in B_1$ and $y\in B_2$
be arbitrary two vertices. Suppose, on the contrary,  that $x$ and $y$ are not adjacent. Note that for any $x_i\in V(P)$, it is impossible that $x\rightarrow x_i$ and $x_{k-2+i}\rightarrow y$ both hold,
otherwise, $xx_ix_{i+1}\ldots x_{k-2+i}y$ implies that $\overline{xy}$, a contradiction, where the subscripts are taken modulo $k+3$. By Lemma \ref{bipartite}(2), if $y\rightarrow x_j$ for some $x_j\in O(P)$, then
$x\Rightarrow \{x_0, \ldots, x_{j-1}\}\cap E(P)$. Analogously,  if $x\rightarrow x_i$ for some $x_i\in E(P)$,
then $y\Rightarrow \{x_0, \ldots, x_{i-1}\}\cap O(P)$.  By Lemma  \ref{Icocbc},
$x$ and every vertex of $\{x_0, x_2, x_{k-1}, x_{k+1}\}$ are adjacent and $y$ and every vertex of $\{x_1, x_3, x_k, x_{k+2}\}$
are adjacent.
Suppose $y\rightarrow x_k$. Then we have $x\Rightarrow \{x_0, x_1,\ldots, x_{k-1}\}\cap E(P)$, in particular, $x\rightarrow x_{k-1}$.
By the definition of $B$ and $B_1\cap B_2=\emptyset$,  $x_{k+1}\rightarrow x$.  So it must be  $x_3\rightarrow y$,
otherwise $yx_3x_4\ldots x_{k+1}x$ implies $\overline{xy}$, a contradiction.
But it is impossible as  $x\rightarrow x_{k-1}$. Therefore  $x_k\mapsto y$. This implies  $x_2\mapsto x$ and
furthermore  $\{x_3, x_5, \ldots, x_k, x_{k+2}\}\Rightarrow y$, in particular, $x_3\rightarrow y$.
By the definition of $B$ and $B_1\cap B_2=\emptyset$, we have  $y\rightarrow x_1$. So $x\rightarrow x_{k-1}$.
But it is impossible as $x_3\rightarrow y$. Thus $\overline{xy}$. By the arbitrariness of $x$ and $y$,
every vertex of $B_1$ and every vertex of $B_2$ are adjacent.

Next we show that $B_1$ and $B_2$ both are independent sets. Suppose not. Without loss of generality,
assume that there exist two vertices $x', x''\in B_1$ such that $x'\rightarrow x''$. By the definition of $B_1$ and $B_1\cap B_2=\emptyset$,
there exists $x_i\in E(P)$ such that $x''\rightarrow x_i$. But the path $x'x''C[x_i, x_{i+(k-2)}]$
implies $\overline{x'x_{i+k-2}}$. Note that $i+k-2$ is odd. This is a contradiction to $x'\in B_1$.
If one of $B_1$ and $B_2$ is empty, then $D[B]$ is an empty digraph, otherwise $D[B]$ is a semicomplete bipartite digraph.
\end{proof}

\begin{theorem}\label{theorem2} The subdigraph induced by $V(D)\setminus V(P)$ is either a semicomplete digraph,
a semicomplete bipartite digraph or an empty digraph.\end{theorem}

\begin{proof} By Theorem \ref{theorem}, $D[V(P)]$ is either a semicomplete digraph or a semicomplete
bipartite digraph. If $D[V(P)]$ is a semicomplete digraph, then by Lemma \ref{sout}, we are done.
Now consider that $D[V(P)]$ is a semicomplete bipartite digraph.  From now on, all subscripts
appearing in this proof are taken modulo $k+3$.

By Lemma \ref{semiDvp}, $D[B]$ is either a semicomplete bipartite digraph or an empty digraph.
By Lemma \ref{Icocbc}, for any $x\in I$, either $x\mapsto E(P)$ or $x\mapsto O(P)$ or both and for any $y\in W$,
either $E(P)\mapsto y$ or $O(P)\mapsto y$ or both. Hence, we divide $I$ into two sets:
$I_1=\{x\in I: x\mapsto E(P)\}$ and $I_2=\{x\in I: x\mapsto O(P)\}$ and divide $W$ into two sets:
$W_1=\{x\in W: E(P)\mapsto x\}$ and $W_2=\{x\in W: O(P)\mapsto x\}$. $B_1$ and $B_2$ are defined similar to Lemma \ref{semiDvp}.
Now we show that $D[I]$ is either a  semicomplete bipartite digraph or an empty digraph. If $|I|\le 1$, there is nothing to prove. Now assume $|I|\ge 2$.
Since $D$ is strong and $(V(P), I)=\emptyset$, we have $B\cup W\neq\emptyset$ and $(B\cup W, I)\neq\emptyset$.
Define $\widetilde{I}=\{x\in I: (B\cup W, x)\neq\emptyset\}$. Clearly, $\widetilde{I}\neq\emptyset$.

To complete the proof of this theorem, we first give the following several claims.

\vskip0.2cm
\noindent{\bf Claim 1.} Every vertex of $I_i$ is adjacent to every vertex of $B_{3-i}\cup W_{3-i}$ and every vertex of $W_i$ is adjacent to every vertex of $B_{3-i}$ for $i=1,2$.

\begin{proof}
Let $x\in I_1$ and $y\in B_2\cup W_2$ be arbitrary. By the definition of $B_2\cup W_2$, there exists $x_i\in O(P)$
such that $x_i\rightarrow y$. Then the path $xC[x_{i+5}, x_i]y$ is a path of length $k$, which implies that $\overline{xy}$.  Analogously, we can show that every vertex of $I_2$
is adjacent to every vertex of $B_1\cup W_1$ and every vertex of $W_i$ is adjacent to every vertex of $B_{3-i}$ for $i=1,2$.
\end{proof}

\vskip0.2cm
\noindent{\bf Claim 2.} Let $Q=y_0y_1\ldots y_q$ be a path of $V(D)\setminus V(P)$ with  $q\le k-1$. Suppose that
there exists $x_i\in V(P)$ such that $x_i\rightarrow y_0$. If $i$ and $q$ have the same parity, then $y_q$
and $E(P)$ are adjacent; if the parity of $i$ and $q$ are different, then $y_q$ and $O(P)$ are adjacent.

\begin{proof} Note that $x_{(q+i)-(k-1)}\ldots x_iy_0\ldots y_q$ is a path of length $k$, which implies
$\overline{x_{(q+i)-(k-1)}y_q}$. If $i$ and $q$ have the same parity, then $(q+i)-(k-1)$ is even and so
$x_{(q+i)-(k-1)}\in E(P)$. If the parity of $i$ and $q$ are different, then $(q+i)-(k-1)$ is odd and so
$x_{(q+i)-(k-1)}\in O(P)$.\end{proof}

Similar to Claim 2, we can obtain the following claim.
\vskip0.2cm

\noindent{\bf Claim 3.} Let $Q=y_0y_1\ldots y_q$ be a path of $V(D)\setminus V(P)$ with  $q\le k-1$.
Suppose that there exists $x_j\in V(P)$ such that $y_q\rightarrow x_j$. If $j$ and $q$ have the same parity,
then $y_0$ and $E(P)$ are adjacent; if $j$ and $q$ have the different parity, then $y_0$ and $O(P)$ are adjacent.

\vskip 0.2cm

\noindent{\bf Claim 4.} $(W, I)=\emptyset$.

\begin{proof} Suppose not. Let $xy$ be an arc from $W$ to $I$. By Lemma \ref{Icocbc}(1) and (2), there exist $x_i\in\{x_0, x_1\}$ and $x_j\in \{x_{k+1}, x_{k+2}\}$ such that $x_i\rightarrow x$ and $y\rightarrow x_j$. Since $P$ is minimal, we have $d(x_i, x_j)\ge k\ge 5$. However $x_ixyx_j$ is a path of length 3, a contradiction.
\end{proof}

\noindent{\bf Claim 5.} $(B_i, I_i)=\emptyset$ and $(W_i, B_i)=\emptyset$, for $i=1,2$.

\begin{proof}
  Assume $(B_1, I_1)\neq\emptyset$. Let $xy$ be an arc from $B_1$ to $I_1$. By Lemma \ref{Icocbc}(1), $y\rightarrow x_{k+1}$. Consider the path $xy$. By Claim 3, $x$  and $O(P)$ are adjacent. Since $B_1\cap B_2=\emptyset$, it is impossible. Thus $(B_1, I_1)=\emptyset$. Analogously, we can show that $(B_2, I_2)=\emptyset$  and $(W_i, B_i)=\emptyset$, for $i=1,2$.
\end{proof}

\noindent{\bf Claim 6.} $I_1\cap I_2=\emptyset$ and $W_1\cap W_2=\emptyset$.

\begin{proof} Suppose $I_1\cap I_2\neq\emptyset$. Let $z\in I_1\cap I_2$ be arbitrary.
By Lemma \ref{Icocbc}(1),
$z\mapsto V(P)$ and by Claim 1, $z$ and every vertex of $B\cup W$ is adjacent.
According to Claims 4 and 5, we can conclude that $z\mapsto B\cup W$.
That is to say, $I_1\cap I_2\mapsto B\cup W$,
which also implies that for any $w\in\widetilde{I}$, $w\notin I_1\cap I_2$. Since $D$ is strong,
there exist $v\in I_1\cap I_2$ and $u\in \widetilde{I}$ such that $u\rightarrow v$.
According to Claim 3 and $v\rightarrow V(P)$, $u\in I_1\cap I_2$, a contradiction.  Hence $I_1\cap I_2=\emptyset$.
Analogously, we can show that $W_1\cap W_2=\emptyset$. \end{proof}

By Claims 3 and 6, we have the following claim.

\vskip 0.2cm

\noindent{\bf Claim 7.} $I_i$ and $W_i$ are both independent sets, for $i=1,2$.
\vskip 0.2cm

\noindent{\bf Claim 8.} $(I_i, B_i)=\emptyset$ and $(B_i, W_i)=\emptyset$, for $i=1,2$.

\begin{proof}
  Assume $(I_1, B_1)\neq\emptyset$. Let $uv$ be an arc from $I_1$ to $B_1$. There exists $x_i\in E(P)$ such that $v\rightarrow x_i$. Then by Claim 4, $u$ and $O(P)$ are adjacent, a contradiction.  Thus $(I_1, B_1)=\emptyset$. Analogously, we can show that $(I_2, B_2)=\emptyset$  and $(B_i, W_i)=\emptyset$, for $i=1,2$.
\end{proof}

\noindent{\bf Claim 9.} For any $w\in I\setminus \widetilde{I}$, $d_{D[I]}(\widetilde{I}, w)<\infty$.

\begin{proof}
Since $D$ is strong, $w$ is reachable from $\widetilde{I}$ in $D$. Let $Q=y_0y_1\ldots y_m$
be a shortest path from $\widetilde{I}$ to $w$ in $D$, where $y_0\in\widetilde{I}$, $y_m=w$
and $m\ge 1$. Now we show $V(Q)\subset\widetilde{I}$. Because $y_m\in I\setminus \widetilde{I}$,
we have $y_{m-1}\in I$. Denote $r=\min \{j: y_j,\ldots, y_m\in I\}$. If $r\ge 1$, then $y_{r-1}\notin I$.
By the definition of $\widetilde{I}$, $y_r\in \widetilde{I}$. Then $y_r\ldots y_m$ is a shorter path from
$\widetilde{I}$ to $w$ than $Q$, a contradiction. Hence $r=0$ and so $V(Q)\subset I$.
\end{proof}

\noindent{\bf Claim 10.} $(I_i, W_i)=\emptyset$ for $i=1,2$.

\begin{proof} Assume $(I_1, W_1)\neq\emptyset$. Let $uv$ be an arc from $I_1$ to $W_1$.
If $u\in \widetilde{I}$, then by Claims 4 and 5, there exists $z\in B_2$ such that $z\rightarrow u$.
There exists $x_i\in O(P)$ such that $x_i\rightarrow z$.  Consider the path $zuv$.  By Claim 3, $v$ and $O(P)$ are adjacent, a contradiction.  Hence $u\notin \widetilde{I}$. By Claim 9, there exists a path $Q=u_0u_1\ldots u_t$ from $\widetilde{I}$ to $u$ in $D[I]$,
where $u_0\in \widetilde{I}$ and $u_t=u$. If $t\ge k-1$, then  $u_{t-(k-1)}\ldots u_tv$ is a path of length $k$,
which implies $\overline{u_{t-(k-1)}v}$ and furthermore $u_{t-(k-1)}\rightarrow v$ by Claim 4. Repeating using this way,
there exists an integer $i_0$ such that $u_{t-i_0(k-1)}\rightarrow v$ and $0\le t-i_0(k-1)<k-1$. Note that $u_{t-i_0(k-1)}\in I_1$. So we assume, without loss of generality, that $t\le k-2$.

If $u_0\in I_1$, then, by Claim 7, $t$ is even and so $t\le k-3$. Since $u_0\in \widetilde{I}$ and $I_1\Rightarrow B_1\cup W$,
there exists $y\in B_2$ such that $y\rightarrow u_0$ and there exists $x_j\in O(P)$
such that $x_j\rightarrow y$. Consider the path $R_1=yu_0\ldots u_tv$.
Note that the length of $R_1$ is even. By Claim 2, $v$ and $O(P)$ is adjacent, a contradiction.
If $u_0\in I_2$, we have $t$ is odd. Since $u_0\in \widetilde{I}$ and $I_2\Rightarrow B_2\cup W$,
there exists $y\in B_1$ such that $y\rightarrow u_0$ and there exists $x_i\in E(P)$
such that $x_i\rightarrow y$. Consider the path $R_2=yu_0\ldots u_tv$. By Claims 5 and 8, we have that $y$ and $v$ are not adjacent and so $t\le k-4$. By Claim 2, $v$ and $O(P)$ is adjacent, a contradiction.
Therefore, $(I_1, W_1)=\emptyset$. Analogously,
we can show $(I_2, W_2)=\emptyset$. \end{proof}

\noindent{\bf Claim 11.} Every vertex of $\widetilde{I}\cap I_i$ is adjacent to every vertex of $I_{3-i}$ for $i=1,2$.

\begin{proof}  Let $x\in \widetilde{I}\cap I_1$ and
$x''\in I_2$ be arbitrary. By Claims 4 and 5, there exists $y\in B_2$ such that $y\rightarrow x$ and furthermore
there exists $x_j\in O(P)$ such that $x_j\rightarrow y$. Then $x''P[x_{j-(k-3)}, x_j]yx$ implies $\overline{x''x}$. Analogously,
we can show that every vertex of $\widetilde{I}\cap I_2$ is adjacent to every vertex of $I_1$.
\end{proof}

Now we return the proof of the theorem. By Claim 7, $I_1$ and $I_2$ are both independent sets. If one of $I_1$ and $I_2$ is an empty set, then $D[I]$ is an empty digraph. Assume that $I_1$ and $I_2$ are both nonempty sets.

Now we show that $D[I]$ is a semicomplete bipartite digraph. Using Claim 11, we only need to prove that $D[I\setminus\widetilde{I}]$ is a semicomplete bipartite digraph.

Let $x'\in I_1\setminus\widetilde{I}$ and $x''\in I_2\setminus\widetilde{I}$ be arbitrary.
We shall show that $\overline{x'x''}$. By Claim 9, $d_{D[I]}(\widetilde{I}, x')<\infty$ and $d_{D[I]}(\widetilde{I}, x'')<\infty$.
Without loss of generality,  assume that $d_{D[I]}(\widetilde{I}, x')\le d_{D[I]}(\widetilde{I}, x'')$. Let $R=z_0z_1\ldots z_m$
be a shortest path from $\widetilde{I}$ to $x'$ in $D[I]$, where $z_0\in \widetilde{I}$ and $z_m=x'$.
By the minimality of $R$, $z_1, \ldots, z_m\in I\setminus \widetilde{I}$. In addition,  $x''\notin V(R)$.
By the definition of $\widetilde{I}$, there exists $y\in B$ such that $y\rightarrow z_0$.

Assume that $z_0\in I_1$. In this case $y\in B_2$ and there exists $x_j\in O(P)$
such that $x_j\rightarrow y$. By Claim 7 and $x'\in I_1$, we have $m$ is even and so $m\ge 2$.
By Claim 11, $\overline{z_0x''}$ and $x''\rightarrow z_0$ as $d_{D[I]}(\widetilde{I}, x')\le d_{D[I]}(\widetilde{I}, x'')$.
If $m\le k-3$, then $x''P[x_{j-(k-3-m)}, x_j]yR$ implies that $\overline{x''z_m}$. If $m=k-1$,
then $x''R$ implies $\overline{x''z_m}$.  The proof for the case $m\ge k-1$ is by induction on even $m$ with
the case $m=k-1$ as the basis. By induction, $x''$ and every vertex of $\{z_0, z_2,\ldots, z_{m-2}\}$ are adjacent,
in particular, $x''$ and $z_{m-(k-1)}$ are adjacent. As $d_{D[I]}(\widetilde{I}, x')\le d_{D[I]}(\widetilde{I}, x'')$,
$x''\rightarrow z_{m-(k-1)}$. Note that $x''R[z_{m-(k-1)}, z_m]$ is a path of length $k$, which implies that
$\overline{x''z_m}$.

Now consider the case $z_0\in I_2$. In this case $y\in B_1$ and $x''\rightarrow y$. There exists $x_j\in E(P)$ such that $x_j\rightarrow y$. According to Claim 7 and $x'\in I_1$, we have $m$ is odd and $m\ge 1$. If $m\le k-3$, then $x''P[x_{j-(k-3-m)}, x_j]yR$
implies that $\overline{x''z_m}$. If $m=k-2$, then $x''yR$  implies that $x''z_m$.
The proof for the case $m\ge k$ is by induction on odd $m$ with the case $m=k-2$ as the basis.
By induction, $x''$ and every vertex of $\{z_1, z_3,\ldots, z_{m-2}\}$ are adjacent,
in particular, $x''$ and $z_{m-(k-1)}$ are adjacent. As $d_{D[I]}(\widetilde{I}, x')\le d_{D[I]}(\widetilde{I}, x'')$,
$x''\rightarrow z_{m-(k-1)}$.  Then the path $x''R[z_{m-(k-1)}, z_m]$ implies $\overline{x''z_m}$.
From now on, we have shown that $D[I]$ is either a semicomplete bipartite digraph or an empty digraph.
Analogously, we can show that $D[W]$ is either a semicomplete bipartite digraph or an empty digraph.
Recalled that $D[B]$ is either a semicomplete bipartite digraph or an empty digraph.
Combining these with Claims 1,5,8 and 10, we have that $D[V(D)\setminus V(P)]$ is either a
semicomplete bipartite digraph or an empty digraph. \end{proof}

From Theorems \ref{theorem} and \ref{theorem2}, Theorem \ref{main result} holds.


\end{document}